\documentclass[11pt, final]{article}
\usepackage{a4}
\usepackage{amsmath}%
\usepackage{amstext}%
\usepackage{amssymb}%
\usepackage{showkeys}%
\usepackage{epsfig}%
\usepackage{cite}
\usepackage{tikz}
\usepackage{subfig}
\usepackage{caption}
\usepackage{multirow}
\usepackage{booktabs}
\usepackage{floatrow}

\usepackage{listings}
\usepackage[margin=1.2in]{geometry}
\newtheorem{theorem}{Theorem}

\newtheorem{algorithm}[theorem]{Algorithm}

\newtheorem{lemma}[theorem]{Lemma}

\newtheorem{pb}[theorem]{Problem}
\newenvironment{problem}{\pb\rm}{\endpb}

\newtheorem{rem}[theorem]{Remark}
\newenvironment{remark}{\rem\rm}{\endrem}

\newcounter{unnumber}

\newtheorem{prf}[unnumber]{Proof.}
\newenvironment{proof}{\prf\rm}{\hfill{$\blacksquare$}\endprf}
\newcommand{\R}{\mathbb{R}}%
\newcommand{\N}{\mathbb{N}}%
\newcommand{\ol}{\overline}%

\DeclareMathOperator*\inte{int}%
\DeclareMathOperator*\sqri{sqri}%
\DeclareMathOperator*\ri{ri}%
\DeclareMathOperator*\dom{dom}%
\DeclareMathOperator*\B{\overline{\R}}%
\DeclareMathOperator*\ran{ran}%
\DeclareMathOperator*\id{Id}%
\DeclareMathOperator*\argmin{argmin}

\title{Fixing and extending some recent results on the ADMM algorithm}

\author{Sebastian Banert \thanks{Lund University, Faculty of Engineering (LTH), Department of Automatic Control, Box 118, 221 00 Lund, Sweden,
email: sebastian.banert@control.lth.se. Research supported by FWF (Austrian Science Fund), project I 2419-N32, and by Wallenberg AI, Autonomous Systems and Software Program (WASP).}
\and Radu Ioan Bo\c{t} \thanks{University of Vienna, Faculty of Mathematics, Oskar-Morgenstern-Platz 1, A-1090 Vienna, Austria,
email: radu.bot@univie.ac.at. Research partially supported by FWF (Austrian Science Fund), project I 2419-N32.} \and
Ern\"{o} Robert Csetnek \thanks {University of Vienna, Faculty of Mathematics, Oskar-Morgenstern-Platz 1, A-1090 Vienna, Austria,
email: ernoe.robert.csetnek@univie.ac.at. Research supported by FWF (Austrian Science Fund), projects M 1682-N25 and P 29809-N32.}}

\begin{document}
\maketitle

\noindent \textbf{Abstract.} We investigate the techniques and ideas used in the convergence analysis of two proximal ADMM algorithms for 
solving convex optimization problems involving compositions with linear operators. Besides this, we formulate a variant of the ADMM algorithm that is able to handle convex optimization problems involving an additional smooth function in its objective, 
and which is evaluated through its gradient. Moreover, in each iteration we allow the use of variable metrics, while the investigations are carried out in the setting of infinite dimensional Hilbert spaces.  This algorithmic scheme is investigated from the 
point of view of its convergence properties.\vspace{1ex}

  \vspace{1ex}


\section{Introduction}\label{sec1}
One of the most popular numerical algorithms for solving optimization problems of the form 
\begin{equation}\label{prim} \inf_{x\in\R^n}\{f(x)+g(Ax)\},
\end{equation}
where  $f:\R^n\rightarrow\B:=\R\cup\{\pm\infty\}$ and $g:\R^m\rightarrow\B$ 
are proper, convex, lower semicontinuous functions and $A:\R^n\rightarrow\R^m$ is a linear operator, 
is the alternating direction method of multipliers (ADMM). Here, the spaces $\R^n$ and $\R^m$ are equipped with their usual inner products and induced norms, which we both denote by $\langle \cdot,\cdot\rangle$ and $\| \cdot \|$, respectively, as there is no risk of confusion.

By introducing an auxiliary variable $z$ one can rewrite \eqref{prim} as  
\begin{equation}\label{prim-x-z}  \inf_{\substack{(x,z)\in\R^n\times\R^m\\Ax-z=0}}\{f(x)+g(z)\}.
\end{equation}

The Lagrangian associated with problem \eqref{prim-x-z} is 
$$l:\R^n \times \R^m \times \R^m \rightarrow \B, \ l(x,z,y)=f(x)+g(z)+\langle y,Ax-z\rangle,$$ 
and we say that $(x^*,z^*,y^*)\in\R^n\times\R^m\times\R^m$ is a saddle point of the Lagrangian, if 
\begin{equation}\label{saddle-point-def} l(x^*,z^*,y)\leq l(x^*,z^*,y^*)\leq l(x,z,y^*) \ \forall (x,z,y)\in\R^n\times\R^m\times\R^m.
\end{equation}

It is known that $(x^*,z^*,y^*)$ is a saddle point of $l$ if and only if $z^*=Ax^*$, $(x^*,z^*)$ is an optimal solution of \eqref{prim-x-z},  
$y^*$ is an optimal solution of the Fenchel dual problem (see \cite{bauschke-book, bo-van, b-hab, EkTem, Zal-carte}) to \eqref{prim}
\begin{equation}\label{dual} \sup_{y\in\R^m}\{-f^*(-A^Ty)-g^*(y)\},\end{equation}
and the optimal objective values of \eqref{prim} and \eqref{dual} coincide. Notice that $f^*$ and $g^*$ are the conjugates of 
$f$ and $g$, defined by $f^*(u)=\sup_{x\in\R^n}\{\langle u,x\rangle-f(x)\}$ for all $u \in \R^n$  and $g^*(y)=\sup_{z\in\R^m}\{\langle y,z\rangle-g(z)\}$ for all $y\in\R^m$, respectively. 

If \eqref{prim} has an optimal solution and $A(\ri(\dom f))\cap\ri\dom g\neq\emptyset$, then the set of saddle points of $l$
is nonempty. Here, we denote by $\ri(S)$ the relative interior of a convex set $S$, which is the interior of $S$ relative to its affine hull. 

For a fixed real number $c >0$ we further consider the augmented Lagrangian associated with problem \eqref{prim-x-z}, which is defined as
$$L_c:\R^n \times \R^m \times \R^m \rightarrow \B, \ L_c(x,z,y)=f(x)+g(z)+\langle y,Ax-z\rangle + \frac{c}{2} \|Ax-z \|^2.$$ 
The ADMM algorithm reads: 

\begin{algorithm}\label{admm} Choose $(z^0,y^0)\in \R^m\times\R^m$ and $c>0$. 
For all $k\geq 0$ generate the sequence $(x^k,z^k,y^k)_{k \geq 0}$ as follows:
 \begin{eqnarray} x^{k+1} & \in & \argmin_{x\in\R^n} L_c(x,z^k,y^k)
 = \argmin_{x\in\R^n} \left\{f(x)+\frac{c}{2}\|Ax-z^k+c^{-1}y^k\|^2 \right\} \label{admm-x}\\
z^{k+1} & = &\argmin_{z\in\R^m} L_c(x^{k+1},z,y^k) = \argmin_{z\in\R^m} \left\{g(z)+\frac{c}{2}\|Ax^{k+1}-z+c^{-1}y^k\|^2 \right\} \label{admm-z}\\
y^{k+1} & = & y^k+c(Ax^{k+1}-z^{k+1}).\label{admm-y}
\end{eqnarray}
\end{algorithm}

If $A$ has full column rank, then the set of minimizers in \eqref{admm-x} is a singleton, as is the set of minimizers in \eqref{admm-z} without any further assumption, and the sequence $(x^k,z^k,y^k)_{k \geq 0}$ generated by Algorithm \ref{admm} converges to a saddle point of the Lagrangian $l$. The alternating direction method of multipliers was first introduced in  \cite{gabaymercier} and \cite{fortinglowinski}. Gabay has shown in \cite{gabay} (see also \cite{ecb}) that ADMM is nothing else than the Douglas-Rachford algorithm applied to the monotone inclusion problem
$$0 \in \partial (f^* \circ (-A^T))(y) + \partial g^*(y) $$
For a proper function $k : \R^n \rightarrow \B$,   the set-valued operator defined by $\partial k(x) := \{u \in \R^n : k(t) - k(x) \geq\langle u,t-x\rangle \ \forall t \in \R^n\}$, for $k(x) \in \R$, and $\partial k(x) := \emptyset$, otherwise, denotes its (convex) subdifferential.

One of the limitations of the ADMM algorithm comes from the presence of the term $Ax$ in the update rule of $x^{k+1}$. While in \eqref{admm-z} a proximal step for the function $g$ is taken, in \eqref{admm-x} the function $f$ and the operator $A$ are not evaluated independently, 
which makes the ADMM algorithm less attractive for implementations than the primal-dual splitting algorithms (see, for instance, 
\cite{b-c-h, b-c-h2, bothendrich, ch-pck, condat2013, vu}). Despite of this fact, the ADMM algorithm has been widely used for solving convex optimization problems arising in real-life applications (see, for instance, \cite{bpcpe, esser}). For a version of the ADMM algorithm with inertial and memory effects we refer the reader to \cite{b-c-inertial-admm}.

In order to overcome the above-mentioned drawback of the classical ADMM method and to increase its flexibility, the following so-called alternating direction proximal method of multipliers  has been considered in \cite{shefi-teboulle2014} (see also \cite{fpst, lst}):

\begin{algorithm}\label{algs-sh-teb} Choose $(x^0,z^0,y^0)\in\R^n\times\R^m\times\R^m$ and $c>0$. 
For all $k\geq 0$ generate the sequence $(x^k,z^k,y^k)_{k \geq 0}$ as follows:
\begin{eqnarray} x^{k+1} & \in &
\argmin_{x\in\R^n} \left\{f(x)+\frac{c}{2}\|Ax-z^k+c^{-1}y^k\|^2 + \frac{1}{2}\|x-x^k\|_{M_1}^2\right\} \label{sh-teb-x}\\
z^{k+1} & = &\argmin_{z\in\R^m} \left\{g(z)+\frac{c}{2}\|Ax^{k+1}-z+c^{-1}y^k\|^2 + \frac{1}{2}\|z-z^k\|_{M_2}^2\right\} \label{sh-teb-z}\\
y^{k+1} & = & y^k+c(Ax^{k+1}-z^{k+1}).\label{sh-teb-y}
\end{eqnarray}
\end{algorithm}

Here, $M_1\in\R^{n\times n}$ and $M_2\in\R^{m\times m}$ are symmetric positive semidefinite matrices and $\|u\|_{M_i}^2=\langle u,M_iu\rangle$ denotes the squared seminorm induced by $M_i$, for $i\in\{1,2\}$. 

Indeed, for $M_1 = M_2=0$, Algorithm \ref{algs-sh-teb} becomes the classical ADMM method, while for
$M_1 = \mu_1 \id$ and $M_2 = \mu_2 \id$ with $\mu_1, \mu_2 >0$ and $\id$  denoting the corresponding matrix, it becomes the algorithm proposed and investigated in \cite{eckstein}. Furthermore, if $M_1 = \tau^{-1}\id - cA^TA$ with $\tau > 0$ 
such that $c\tau \|A\|^2 < 1$ and $M_2=0$, then one can show that Algorithm \ref{algs-sh-teb} is equivalent to one of the primal-dual algorithms formulated in \cite{condat2013}.

The sequence $(z^k)_{k\geq 0}$ generated in Algorithm \ref{algs-sh-teb}  is uniquely determined due to the fact that the objective function in 
\eqref{sh-teb-z} is lower semicontinuous and strongly convex. On the other hand, the set of minimizers in \eqref{sh-teb-x} is in general not a singleton and it can be even empty. However, if one imposes that $M_1 + A^*A$ is positive definite, then $(x^k)_{k\geq 0}$ will be uniquely determined, too. 

Shefi and Teboulle provide in \cite{shefi-teboulle2014} in connection to  Algorithm \ref{algs-sh-teb} an ergodic convergence rate result for a primal-dual gap function formulated in terms of the Lagrangian $l$,
from which they deduce a global convergence rate result for the sequence of function values $(f(x^k) + g(Ax^k))_{k \geq 0}$ to the optimal objective value of \eqref{prim}, when $g$ is Lipschitz continuous. Furthermore, they formulate a global  convergence rate result for the sequence $(\|Ax^k-z^k\|)_{k \geq 0}$ to $0$. Finally, Shefi and Teboulle prove the convergence of the sequence $(x^k,z^k,y^k)_{k \geq 0}$ to a saddle point of the Lagrangian $l$, provided that either $M_1=0$ and $A$ has full column rank or $M_1$ is positive definite.

Algorithm \ref{algs-sh-teb} from \cite{shefi-teboulle2014} represents the starting point of our investigations. More precisely, in this paper:

$\bullet$ we point out some flaws in the proof of a statement in \cite{shefi-teboulle2014}, which is fundamental for the derivation of the global convergence rate of $(\|Ax^k-z^k\|)_{k \geq 0}$ to $0$ and of the convergence of the sequence $(x^k,z^k,y^k)_{k \geq 0}$;

$\bullet$ we show how the statement in cause can be proved by using different techniques;

$\bullet$ we formulate a variant of Algorithm \ref{algs-sh-teb} for solving convex optimization problems  in infinite dimensional Hilbert spaces involving an additional smooth function in their objective, that we evaluate through its gradient, and which allows in each iteration the use of variable metrics;

$\bullet$ we prove  an ergodic convergence rate result for this algorithm involving a primal-dual gap function formulated in terms of the associated Lagrangian $l$ and a convergence result for the sequence of iterates to a saddle point of $l$.

\section{Fixing some results from \cite{shefi-teboulle2014} related to the convergence analysis for Algorithm \ref{algs-sh-teb}}\label{sec2}

In this section we point out several flaws that have been made in \cite{shefi-teboulle2014} when deriving a fundamental result for both the rate of convergence of the sequence  $(\|Ax^k-z^k\|)_{k \geq 0}$ to $0$ and the convergence of the sequence $(x^k,z^k,y^k)_{k \geq 0}$ to a saddle point of the Lagrangian $l$. We also show how these arguments can be fixed by relying on some of the building blocks of the analysis we will carry out in Section \ref{sec3}.

To proceed, we first recall some results from \cite{shefi-teboulle2014}. 
We start with a statement that follows from the variational characterization of the minimizers of \eqref{sh-teb-x}-\eqref{sh-teb-z}.  

\begin{lemma}(see \cite[Lemma 4.2]{shefi-teboulle2014}) \label{l42-sh-teb} Let $(x^k,z^k,y^k)_{k\geq 0}$ be a sequence generated by 
Algorithm \ref{algs-sh-teb}. Then for all $k\geq 0$ and for all $(x,z,y)\in\R^n\times\R^m\times\R^m$ it holds
\begin{eqnarray*}l(x^{k+1},z^{k+1},y)& \leq & l(x,z,y^{k+1})+c\langle z^{k+1}-z^k,A(x-x^{k+1})\rangle+\\
&&+\frac{1}{2}\left(\|x-x^k\|_{M_1}^2-\|x-x^{k+1}\|_{M_1}^2+\|z-z^k\|_{M_2}^2-\|z-z^{k+1}\|_{M_2}^2\right)\\
&&+\frac{1}{2}\left(c^{-1}\|y-y^k\|^2-c^{-1}\|y-y^{k+1}\|^2\right)\\
&&-\frac{1}{2}\left(\|x^{k+1}-x^k\|_{M_1}^2+\|z^{k+1}-z^k\|_{M_2}^2+c^{-1}\|y^{k+1}-y^k\|^2\right).
\end{eqnarray*}
\end{lemma}

Furthermore, by invoking the monotonicity of the convex subdifferential of $g$, in \cite{shefi-teboulle2014} the following estimation is derived.

\begin{lemma}(see \cite[Proposition 5.3(b)]{shefi-teboulle2014}) \label{p53-sh-teb} Let $(x^k,z^k,y^k)_{k\geq 0}$ be a sequence generated 
by Algorithm \ref{algs-sh-teb}. Then for all $k\geq 1$ and for all $(x,z)\in\R^n\times\R^m$ it holds
\begin{align*}
c\langle z^{k+1}-z^k,A(x-x^{k+1})\rangle\leq  & \frac{c}{2}\left(\|z-z^k\|^2-\|z-z^{k+1}\|^2+\|Ax-z\|^2\right)+\\
& \frac{1}{2}\left(\|z^{k-1}-z^k\|_{M_2}^2-\|z^k-z^{k+1}\|_{M_2}^2\right).
\end{align*}
\end{lemma}

By taking $(x,z,y):=(x^*,z^*,y^*)$ in Lemma \ref{l42-sh-teb}, where $(x^*,z^*,y^*)$ is a saddle point of the Lagrangian $l$, and  by using 
the inequality (see \eqref{saddle-point-def}) $$l(x^{k+1},z^{k+1},y^*)\geq l(x^*,z^*,y^{k+1}) \ \forall k\geq 0,$$ and the estimation 
in Lemma \ref{p53-sh-teb}, one easily obtains the following result.

\begin{lemma} \label{l51-sh-teb} Let $(x^*,z^*,y^*)$ be a saddle point of the Lagrangian $l$ associated with \eqref{prim}, $M_1,M_2$ be symmetric positive semidefinite matrices and $c>0$. Let $(x^k,z^k,y^k)_{k\geq 0}$ be a sequence generated 
by Algorithm \ref{algs-sh-teb}. Then for all $k\geq 1$ the following inequality holds
\begin{align}
& \ \|x^{k+1}-x^k\|_{M_1}^2+\|z^{k+1}-z^k\|^2_{M_2}+c^{-1}\|y^{k+1}-y^k\|^2 + \\
&\ \|x^*-x^{k+1}\|_{M_1}^2+\|z^*-z^{k+1}\|_{M_2+cI_m}^2+c^{-1}\|y^*-y^{k+1}\|^2+\|z^{k+1}-z^k\|_{M_2}^2\\
\leq & \ \|x^*-x^k\|_{M_1}^2+\|z^*-z^k\|_{M_2+cI_m}^2+c^{-1}\|y^*-y^k\|^2+\|z^k-z^{k-1}\|_{M_2}^2.\label{eql5}
\end{align}
\end{lemma}
By using the notations from \cite[Section 5.3]{shefi-teboulle2014}, namely 
$$v^{k+1}:=\|x^{k+1}-x^k\|_{M_1}^2+\|z^{k+1}-z^k\|^2_{M_2+c\id}+c^{-1}\|y^{k+1}-y^k\|^2 \ \forall k \geq 0$$
and $$u^k:=\|x^*-x^k\|_{M_1}^2+\|z^*-z^k\|_{M_2+c\id}^2+c^{-1}\|y^*-y^k\|^2+\|z^k-z^{k-1}\|_{M_2}^2 \ \forall k \geq 1,$$
the inequality stated in Lemma \ref{l51-sh-teb} can be equivalently written as 
\begin{equation}\label{v} v^{k+1}-c\|z^{k+1}-z^k\|^2\leq u^k-u^{k+1} \ \forall k\geq 1.
\end{equation}
However, in \cite[Lemma 5.1, (5.37)]{shefi-teboulle2014}, instead of \eqref{v}, it is without proof assumed that
\begin{equation}\label{correct}
v^{k+1}\leq u^k-u^{k+1} \ \forall k\geq 1.
\end{equation}
Since the the sequence $(v^k)_{k \geq 0}$ is monotonically decreasing, statement \eqref{correct}, in combination with straightforward telescoping arguments, leads to the fact that $(v^k)_{k \geq 0}$ converges to zero with a rate of convergence of $\cal{O}$$(1/\sqrt{k})$. This implies that $(\|Ax^k-z^k\|)_{k\geq 0}$ converges to zero with a rate of convergence of $\cal{O}$$(1/\sqrt{k})$ (see \cite[Theorem 5.4]{shefi-teboulle2014}). In addition, statement \eqref{v} is used in \cite[Theorem 5.6]{shefi-teboulle2014} to prove the convergence of the sequence $(x^k,z^k,y^k)_{k\geq 0}$ to a saddle point of the Lagrangian $l$. However, the techniques used in \cite{shefi-teboulle2014}, involving function values and the saddle point inequality, do not lead to \eqref{correct}, but to the weaker inequalitiy \eqref{v}. 

In the following we will show that one can in fact derive \eqref{correct},  however, to this end one needs to use different techniques. These are described in detail in the next section; here we will just show how do they lead to \eqref{correct}. 
We would like to notice that, differently from \cite{shefi-teboulle2014}, in our analysis we will only use properties related to the fact that the convex subdifferential of a proper, convex and lower semicontinuous function is a monotone set-valued operator.

We start our analysis with relation \eqref{fey-var-h0}, which in case $h=0$, $L=0$, $M_1^k=M_1 \succcurlyeq 0$ and $M_2^k=M_2 \succcurlyeq 0$ for all $k \geq 0$ (see the setting of Section \ref{sec3}) reads
\begin{align}
 c\|z^k-Ax^{k+1}\|^2 + \|x^k-x^{k+1}\|_{M_1}^2 + \|z^k-z^{k+1}\|_{M_2}^2 & \leq \nonumber\\
\|x^k-x^*\|_{M_1}^2+\|z^k-Ax^*\|_{M_2+c\id}^2+\frac{1}{c}\|y^k-y^*\|^2 & \nonumber\\ 
- \left(\|x^{k+1}-x^*\|_{M_1}^2+\|z^{k+1}-Ax^*\|_{M_2+c\id}^2+\frac{1}{c}\|y^{k+1}-y^*\|^2\right).&  
 \label{fey-var-h00}
\end{align}
for all $k \geq 0$. Using that
\begin{align*}
 c\|z^k-Ax^{k+1}\|^2 = \ & c \left \|z^k - z^{k+1} - \frac{1}{c}(y^{k+1} - y^k) \right \|^2\\
 = \ & c\|z^k-z^{k+1}\|^2 + \frac{1}{c} \|y^{k+1} - y^k\|^2 + 2 \langle z^{k+1}- z^k, y^{k+1} - y^k\rangle,
\end{align*}
we obtain from  \eqref{fey-var-h00} that
\begin{align}
 2 \langle z^{k+1}- z^k, y^{k+1} - y^k\rangle + \|x^k-x^{k+1}\|_{M_1}^2 + \|z^k-z^{k+1}\|_{M_2+c\id}^2 + \frac{1}{c} \|y^{k+1} - y^k\|^2 & \leq \nonumber\\
\|x^k-x^*\|_{M_1}^2+\|z^k-Ax^*\|_{M_2+c\id}^2+\frac{1}{c}\|y^k-y^*\|^2 & \nonumber\\ 
- \left(\|x^{k+1}-x^*\|_{M_1}^2+\|z^{k+1}-Ax^*\|_{M_2+c\id}^2+\frac{1}{c}\|y^{k+1}-y^*\|^2\right)&  
 \label{fey-var-h01}
\end{align}
for all $k \geq 0$. By taking into account that, according to \eqref{coniter},
\begin{align*}
\langle z^{k+1}-z^k,y^{k+1}-y^k\rangle & \geq \frac{1}{2}\|z^{k+1}-z^k\|_{M_2}^2-\frac{1}{2}\|z^{k}-z^{k-1}\|_{M_2}^2
\end{align*}
for all $k \geq 1$, it yields
\begin{align*}
 \|x^k-x^{k+1}\|_{M_1}^2 + \|z^k-z^{k+1}\|_{M_2+c\id}^2 + \frac{1}{c} \|y^{k+1} - y^k\|^2 & \leq \nonumber\\
\|x^k-x^*\|_{M_1}^2+\|z^k-Ax^*\|_{M_2+c\id}^2+\frac{1}{c}\|y^k-y^*\|^2 + \|z^k - z^{k-1}\|^2_{M_2} & \nonumber\\ 
- \left(\|x^{k+1}-x^*\|_{M_1}^2+\|z^{k+1}-Ax^*\|_{M_2+c\id}^2+\frac{1}{c}\|y^{k+1}-y^*\|^2 +  \|z^{k+1} - z^{k}\|^2_{M_2}\right)&,  
\end{align*}
which is nothing else than \eqref{correct}.

From here, by using that $v^{k+1} \leq v^k$ for all $k \geq 0$ and straightforward telescoping arguments, it follows immediately that $(\|Ax^k-z^k\|)_{k\geq 0}$ converges to zero with a rate of $\cal{O}$$(1/\sqrt{k})$.

We will see in the following section that the inequality \eqref{fey-var-h0} will play an essential role also in the convergence analysis of the sequence of iterates. When applied to the particular context of the optimization problem \eqref{prim} and Algorithm \ref{admm}, Theorem \ref{cong-it} provides a rigorous formulation and a correct and clear proof of the convergence result stated in \cite[Theorem 5.6]{shefi-teboulle2014}.

\section{A variant of the ADMM algorithm in the presence of a smooth function and by involving variable metrics}\label{sec3}

In this section we propose an extension of the ADMM algorithm considered in \cite{shefi-teboulle2014} that we also investigate from the perspective of its convergence properties. This extension is twofold: 
on the one hand, we consider an additional convex differentiable function in the objective of the optimization problem \eqref{prim}, which is evaluated in the algorithm through its gradient, and on the other hand, instead of fixed matrices $M_1,M_2$, we use different matrices in each iteration. Furthermore, we change the setting to infinite dimensional Hilbert spaces. 
We start by describing the problem under investigation:

\begin{problem}\label{admm-p1} Let ${\cal H}$ and ${\cal G}$ be real Hilbert spaces, $f:{\cal H}\rightarrow\B$, 
$g: \cal G\rightarrow\B$ be proper, convex and lower semicontinuous functions, $h:{\cal H}\rightarrow\R$ a convex and 
Fr\'{e}chet differentiable function with $L$-Lipschitz continuous gradient (where $L\geq 0$) and 
$A:{\cal H}\rightarrow{\cal G}$ a linear continuous operator. The Lagrangian associated with the convex optimization problem 
\begin{equation}\label{prim-h} \inf_{x\in{\cal H}}\{f(x)+h(x)+g(Ax)\}
\end{equation}
is 
$$l: {\cal H} \times {\cal H} \times {\cal G} \rightarrow \overline \R, \  l(x,z,y)=f(x)+h(x)+g(z)+\langle y,Ax-z\rangle.$$
We say that 
$(x^*,z^*,y^*)\in{\cal H}\times{\cal G}\times{\cal G}$ is a saddle point of the Lagrangian $l$, if the following inequalities hold 
\begin{equation}\label{saddle-point-def-h} l(x^*,z^*,y)\leq l(x^*,z^*,y^*)\leq l(x,z,y^*) \ \forall (x,z,y)\in{\cal H}\times{\cal G}\times{\cal G}.
\end{equation}

Notice that $(x^*,z^*,y^*)$ is a saddle point if and only if $z^*=Ax^*$, $x^*$ is an optimal solution of \eqref{prim-h},  
$y^*$ is an optimal solution of the Fenchel dual problem to \eqref{prim-h}
\begin{equation}\label{dual-h} (D') \ \ \ \sup_{y\in{\cal G}}\{-(f^*\mathop{\Box} h^*)(-A^*y)-g^*(y)\},\end{equation}
and the optimal objective values of \eqref{prim-h} and \eqref{dual-h} coincide, 
where $A^*:{\cal G}\rightarrow{\cal H}$ is the adjoint operator defined by $\langle A^*v,x\rangle=\langle v, Ax\rangle$ for all 
$(v,x)\in {\cal G}\times{\cal H}$. The infimal convolution $f^*\mathop{\Box} h^*:{\cal H}\rightarrow\B$ is defined by 
$(f^*\mathop{\Box} h^*)(x)=\inf_{u\in {\cal H}}\{f^*(u)+h^*(x-u)\}$ for all $x\in {\cal H}$.
\end{problem}

For the reader's convenience, we discuss some situations which lead to the existence of saddle points. This is for instance the case when \eqref{prim-h} has an optimal solution and the Attouch-Br\'{e}zis qualification condition
\begin{equation}\label{reg-cond} 0\in\sqri(\dom g-A(\dom f))
\end{equation}
holds.
Here, for a convex set  $S\subseteq {\cal G}$,  we denote by
$$\sqri S:=\{x\in S:\cup_{\lambda>0}\lambda(S-x) \ \mbox{is a closed linear subspace of} \ {\cal G}\}$$
its strong quasi-relative interior. Notice that the classical interior is contained in the 
strong quasi-relative interior: $\inte S\subseteq\sqri S$, however, in general this inclusion may be strict. If ${\cal G}$ is finite-dimensional, then for a nonempty and convex set $S \subseteq   {\cal G}$, one has $\sqri S =\ri S$. Considering again the 
infinite dimensional setting, we remark that condition 
\eqref{reg-cond} is fulfilled if there exists $x'\in\dom f$ such that $Ax'\in \dom g$ and $g$ is continuous at $Ax'$. 

The optimality conditions for the primal-dual pair of optimization problems \eqref{prim-h}-\eqref{dual-h} read
\begin{equation}\label{opt-cond} -A^*y-\nabla h(x)\in\partial f(x) \mbox{ and } y\in\partial g(Ax). 
\end{equation}
This means that if  \eqref{prim-h} has an optimal solution $x^*\in{\cal H}$ and the qualification condition \eqref{reg-cond} is fulfilled, 
then there exists $y^*\in{\cal G}$, an optimal solution of \eqref{dual-h}, such that  \eqref{opt-cond} holds and $(x^*,Ax^*,y^*)$ is a saddle point of the Lagrangian $l$. Conversely, if the pair 
$(x^*,y^*)\in{\cal H}\times{\cal G}$ satisfies relation \eqref{opt-cond}, then $x^*$ is an optimal solution to  \eqref{prim-h}, 
$y^*$ is an optimal solution to \eqref{dual-h} and $(x^*,Ax^*,y^*)$ is a saddle 
point of the Lagrangian $l$. For further considerations on convex duality  we invite the reader to consult 
\cite{bo-van, b-hab, bauschke-book, EkTem, Zal-carte}.

Furthermore, we discuss some conditions ensuring that \eqref{prim-h} has an optimal solution. 
Suppose that \eqref{prim-h} is feasible, which means that its optimal objective
value is not $+\infty$. The existence of optimal solutions to \eqref{prim-h} is guaranteed if,
for instance, $f+h$ is coercive (that is $\lim_{\|x\|\rightarrow\infty}(f+h)(x)=+\infty$)
and $g$ is bounded from below. Indeed, under these circumstances, the objective function of
\eqref{prim-h} is coercive and the statement follows via \cite[Corollary 11.15]{bauschke-book}. 
On the other hand, if $f+h$ is strongly convex, then the objective function of
\eqref{prim-h} is strongly convex, too, thus \eqref{prim-h} has a unique optimal solution (see \cite[Corollary 11.16]{bauschke-book}).

Some more notations are in order before we state the algorithm for solving Problem \ref{admm-p1}. We denote by ${\cal S_+}({\cal H})$ the family of operators $U:{\cal H}\rightarrow {\cal H}$ which are 
linear, continuous, self-adjoint and positive semidefinite. For $U\in {\cal S_+}({\cal H})$ we consider the semi-norm defined by 
$$\|x\|_U^2=\langle x,Ux\rangle \ \forall x\in {\cal H}.$$
We also make use of the Loewner partial ordering defined for $U_1,U_2\in {\cal S_+}({\cal H})$ by
$$U_1\succcurlyeq U_2\Leftrightarrow \|x\|_{U_1}^2\geq \|x\|_{U_2}^2  \ \forall x\in {\cal H}.$$
Finally, for $\alpha> 0$, we set  
$${\cal P}_{\alpha}({\cal H})=\{U\in {\cal S_+}({\cal H}): U\succcurlyeq \alpha\id \}.$$

\begin{algorithm}\label{alg-h-var} Let $M_1^k\in {\cal S_+}({\cal H})$ and $M_2^k\in {\cal S_+}({\cal G})$ for all $k\geq 0$. 
Choose $(x^0,z^0,y^0)\in{\cal H}\times{\cal G}\times{\cal G}$ and $c>0$. For all $k\geq 0$ generate the sequence $(x^k,z^k,y^k)_{k \geq 0}$ as follows:
\begin{eqnarray} x^{k+1} & \in &
\argmin_{x\in{\cal H}} \left\{f(x)+\langle x-x^k,\nabla h(x^k)\rangle +\frac{c}{2}\|Ax-z^k+c^{-1}y^k\|^2 + \frac{1}{2}\|x-x^k\|_{M_1^k}^2\right\} \label{h-var-x}\\
z^{k+1} & = &\argmin_{z\in{\cal G}} \left\{g(z)+\frac{c}{2}\|Ax^{k+1}-z+c^{-1}y^k\|^2 + \frac{1}{2}\|z-z^k\|_{M_2^k}^2\right\} \label{h-var-z}\\
y^{k+1} & = & y^k+c(Ax^{k+1}-z^{k+1}).\label{h-var-y}
\end{eqnarray}
\end{algorithm}

\begin{remark} (i) If $h=0$ and $M_1^k=M_1$, $M_2^k=M_2$ are constant in each iteration, then Algorithm \ref{alg-h-var} becomes  Algorithm \ref{algs-sh-teb}, which has been investigated in \cite{shefi-teboulle2014}. 

(ii) In order to ensure that the sequence $(x^k)_{k\geq 0}$ is uniquely determined one can assume that for all $k\geq 0$ 
there exists $\alpha_1^k>0$ such that $M_1^k + cA^*A \in {\cal P}_{\alpha_1^k}({\cal H})$.  

This is in particular  the case when
\begin{equation}\label{h}\exists \alpha>0 \mbox{ such that } A^*A\in {\cal P}_{\alpha}({\cal H}). 
\end{equation}
Relying on \cite[Fact 2.19]{bauschke-book}, on can see 
that \eqref{h} holds if and only if $A$ is injective and $\ran A^*$ is closed. Hence, in finite dimensional spaces, namely, if ${\cal H}=\R^n$ and 
${\cal G}=\R^m$, with $m\geq n\geq 1$, \eqref{h}  is nothing else than saying that $A$ has full column rank.

(iii) One of the pioneering works addressing proximal ADMM algorithms in Hilbert spaces, in the particular case when $h=0$ and $M_1^k$ and $M_2^k$ are equal for all $k \geq 0$ to the corresponding identity operators, is the paper by Attouch and Soueycatt \cite{as}. We also refer the reader to \cite{fpst, lst} for versions of the proximal ADMM algorithm stated in finite-dimensional spaces and with proximal terms induced by constant linear operators.
\end{remark}

\begin{remark}\label{rem-cond} We show that the particular choices $M_1^k=\frac{1}{\tau_k}\id-cA^*A$, for $\tau_k > 0$, and $M_2^k=0$ for all $k\geq 0$ lead to a primal-dual algorithm introduced in \cite{condat2013}.
Here $\id : {\cal H} \rightarrow {\cal H}$ denotes the identity operator on ${\cal H}$. Let $k \geq 0$ be fixed.
The optimality condition for \eqref{h-var-x} reads (for $x^{k+2}$): 
\begin{eqnarray*}0&\in&\partial f(x^{k+2})+cA^*(Ax^{k+2}-z^{k+1}+c^{-1}y^{k+1})+M_1^{k+1}(x^{k+2}-x^{k+1})+\nabla h(x^{k+1})\\
& = & \partial f(x^{k+2})+(cA^*A+M_1^{k+1})x^{k+2}+cA^*(-z^{k+1}+c^{-1}y^{k+1})-M_1^{k+1}x^{k+1}+\nabla h(x^{k+1}). 
\end{eqnarray*}
From \eqref{h-var-y} we have $$cA^*(-z^{k+1}+c^{-1}y^{k+1})=A^*(2y^{k+1}-y^k)-cA^*A x^{k+1},$$
hence 
\begin{equation} 0\in\partial f(x^{k+2})+(cA^*A+M_1^{k+1})(x^{k+2}-x^{k+1})+A^*(2y^{k+1}-y^k)+\nabla h(x^{k+1}).
\end{equation}
By taking into account the special choice of $M_1^k$ we obtain 
$$0\in\partial f(x^{k+2})+\frac{1}{\tau_{k+1}}\left(x^{k+2}-x^{k+1}\right)+A^*(2y^{k+1}-y^k)+\nabla h(x^{k+1}),$$
thus,
\begin{eqnarray}\!\!x^{k+2} \!\!\! & = & \!\! (\id+\tau_{k+1}\partial f)^{-1}\left(x^{k+1}-\tau_{k+1}\nabla h(x^{k+1})-\tau_{k+1}A^*(2y^{k+1}-y^k)\right)\nonumber\\
\!\!\!& = & \!\! \argmin_{x\in {\cal H}}\left\{f(x)+\frac{1}{2\tau_{k+1}}\left\|x-\left(x^{k+1}-\tau_{k+1}\nabla h(x^{k+1})
-\tau_{k+1}A^*(2y^{k+1}-y^k)\right)\right\|^2\right\}. \label{cond2}
\end{eqnarray}
Furthermore, from the optimality condition for \eqref{h-var-z} we obtain 
\begin{equation}\label{opt-cond-g-rem} c(Ax^{k+1}-z^{k+1}+c^{-1}y^k)+M_2^k(z^k-z^{k+1})\in\partial g(z^{k+1}),
\end{equation}
which combined with \eqref{h-var-y} gives 
\begin{equation}\label{opt-cond-g2-rem} y^{k+1}+M_2^k(z^k-z^{k+1})\in\partial g(z^{k+1}).
\end{equation}
Using that $M_2^k=0$ and again  \eqref{h-var-y}, it further follows 
\begin{eqnarray*}0 &\in& \partial g^*(y^{k+1})-z^{k+1}\\
&=& \partial g^*(y^{k+1})+c^{-1}(y^{k+1}-y^k-cAx^{k+1}),
\end{eqnarray*}
which is equivalent to
\begin{eqnarray}y^{k+1} & = &(\id+c\partial g^*)^{-1}\left(y^k+cAx^{k+1}\right)\nonumber\\
& = & \argmin_{z\in {\cal G}}\left\{g^*(z)+\frac{1}{2c}\left\|z-\left(y^k+cAx^{k+1}\right)\right\|^2\right\}. \label{cond1}
\end{eqnarray}

The iterative scheme obtained in \eqref{cond1} and \eqref{cond2} generates, for a given starting point $(x^1,y^0)\in{\cal H} \times{\cal G}$ and $c>0$, 
the sequence $(x^k,y^k)_{k \geq 1}$ for all $k \geq 0$ as follows
\begin{eqnarray*}
y^{k+1} & = & \argmin_{z\in {\cal G}}\left\{g^*(z)+\frac{1}{2c}\left\|z-\left(y^k+cAx^{k+1}\right)\right\|^2\right\}\\
x^{k+2} & = & \argmin_{x\in {\cal H}}\left\{f(x)+\frac{1}{2\tau_{k+1}}\left\|x-\left(x^{k+1}-\tau_{k+1}\nabla h(x^{k+1})
-\tau_{k+1}A^*(2y^{k+1}-y^k)\right)\right\|^2\right\}.
\end{eqnarray*}
For  $\tau_k=\tau >0$ for all $k\geq 1$ one recovers a primal-dual algorithm from \cite{condat2013} that has been investigated 
under the assumption $\frac{1}{\tau}-c\|A\|^2>\frac{L}{2}$ (see Algorithm 3.2 and Theorem 3.1 in \cite{condat2013}). We invite the reader to consult 
\cite{ch-pck, vu, b-c-h, b-c-h2} for more insights into primal-dual algorithms and their highlights. Primal-dual algorithms with dynamic step sizes have been investigated in \cite{ch-pck} and \cite{b-c-h2}, 
where it has been shown that clever strategies in the choice of the step sizes can improve the convergence behavior.
\end{remark}

\subsection{Ergodic convergence rates for the primal-dual gap}\label{subsec31}

In this section we will provide a convergence rate result for a primal-dual gap function formulated in terms of the associated Lagrangian $l$.
We start by proving a technical statement (see also \cite{shefi-teboulle2014}). 

\begin{lemma}\label{l42-sh-teb-h} In the context of Problem \ref{admm-p1}, let $(x^k,z^k,y^k)_{k\geq 0}$ 
be a sequence generated by Algorithm \ref{alg-h-var}. Then for all $k\geq 0$ and all $(x,z,y)\in{\cal H}\times{\cal G}\times{\cal G}$ the following inequality holds
\begin{align*}l(x^{k+1},z^{k+1},y) \leq & \ l(x,z,y^{k+1})+c\langle z^{k+1}-z^k,A(x-x^{k+1})\rangle\\
&+\frac{1}{2}\left(\|x-x^k\|_{M_1^k}^2 + \|z-z^k\|_{M_2^k}^2 + c^{-1}\|y-y^k\|^2\right)\\
&-\frac{1}{2}\left(\|x-x^{k+1}\|_{M_1^k}^2+\|z-z^{k+1}\|_{M_2^k}^2+c^{-1}\|y-y^{k+1}\|^2\right)\\
&-\frac{1}{2}\left(\|x^{k+1}-x^k\|_{M_1^k}^2-L\|x^{k+1}-x^k\|^2+\|z^{k+1}-z^k\|_{M_2^k}^2+c^{-1}\|y^{k+1}-y^k\|^2\right).
\end{align*}
Moreover, we have for all $k\geq 0$
$$c\langle z^{k+1}-z^k,A(x-x^{k+1})\rangle\leq \frac{c}{2}\left(\|Ax-z^k\|^2-\|Ax-z^{k+1}\|^2\right)+\frac{1}{2c}\|y^{k+1}-y^k\|^2.$$
\end{lemma}

\begin{proof} We fix $k\geq 0$ and $(x,z,y)\in{\cal H}\times{\cal G}\times{\cal G}$. 
Writing the optimality conditions for \eqref{h-var-x} we obtain 
\begin{equation}\label{opt-cond-f} -\nabla h(x^k)+cA^*(z^k-c^{-1}y^k-Ax^{k+1})+M_1^k(x^k-x^{k+1})\in\partial f(x^{k+1}).
\end{equation}
From the definition of the convex subdifferential we derive 
\begin{eqnarray}f(x^{k+1})-f(x)& \leq & \langle \nabla h(x^k)+cA^*(-z^k+c^{-1}y^k+Ax^{k+1})+M_1^k(-x^k+x^{k+1}),x-x^{k+1} \rangle\nonumber\\
&=& \langle \nabla h(x^k),x-x^{k+1}\rangle+\langle y^{k+1},A(x-x^{k+1})\rangle-c\langle z^k-z^{k+1},A(x-x^{k+1})\rangle\label{ineq-f1}\nonumber\\ 
&&+\langle M_1^k(x^{k+1}-x^k),x-x^{k+1}\rangle\label{ineq-f2},
\end{eqnarray}
where for the last equality we used \eqref{h-var-y}. 

Furthermore, we claim that 
\begin{equation}\label{ineq-h} h(x^{k+1})-h(x)\leq -\langle \nabla h(x^k),x-x^{k+1}\rangle+\frac{L}{2}\|x^{k+1}-x^k\|^2.
\end{equation}
Indeed, this follows by applying the convexity of $h$ and the descent lemma (see \cite[Theorem 18.15 (iii)]{bauschke-book}): 
\begin{align*}
h(x)-h(x^{k+1})-\langle\nabla h(x^k),x-x^{k+1}\rangle & \geq \\
h(x^k)+\langle \nabla h(x^k),x-x^k\rangle-h(x^{k+1})-\langle \nabla h(x^k),x-x^{k+1}\rangle & = \\
h(x^k)-h(x^{k+1})+\langle\nabla h(x^k),x^{k+1}-x^k\rangle & \geq -\frac{L}{2}\|x^{k+1}-x^k\|^2.
\end{align*}
By combining 
\eqref{ineq-f2} and \eqref{ineq-h} we obtain 
\begin{align}(f+h)(x^{k+1}) \leq & \ (f+h)(x)+\langle y^{k+1},A(x-x^{k+1})\rangle-c\langle z^k-z^{k+1},A(x-x^{k+1})\rangle \nonumber\\
& +\frac{1}{2}\|x-x^k\|_{M_1^k}^2-\frac{1}{2}\|x-x^{k+1}\|_{M_1^k}^2-\frac{1}{2}\|x^{k+1}-x^k\|_{M_1^k}^2+\frac{L}{2}\|x^{k+1}-x^k\|^2\label{ineq-f+h2}.
\end{align}
From the optimality condition for \eqref{h-var-z} we obtain 
\begin{equation}\label{opt-cond-g} c(Ax^{k+1}-z^{k+1}+c^{-1}y^k)+M_2^k(z^k-z^{k+1})\in\partial g(z^{k+1}),
\end{equation}
which, combined with \eqref{h-var-y}, gives 
\begin{equation}\label{opt-cond-g2} y^{k+1}+M_2^k(z^k-z^{k+1})\in\partial g(z^{k+1}).
\end{equation}
From here we derive the inequality 
\begin{align} g(z^{k+1})-g(z) \leq & \ \langle -y^{k+1}+M_2^k(z^{k+1}-z^k),z-z^{k+1}\rangle\nonumber\\
= & -\langle y^{k+1},z-z^{k+1}\rangle+\frac{1}{2}\|z-z^k\|_{M_2^k}^2-\frac{1}{2}\|z-z^{k+1}\|_{M_2^k}^2-\frac{1}{2}\|z^{k+1}-z^k\|_{M_2^k}^2.\label{ineq-g}
\end{align}
The first statement of the lemma follows by combining the inequalities \eqref{ineq-f+h2} and \eqref{ineq-g} with the identity (see \eqref{h-var-y})
$$\langle y,Ax^{k+1}-z^{k+1}\rangle= \langle y^{k+1},Ax^{k+1}-z^{k+1}\rangle+\frac{1}{2c}\left(\|y-y^k\|^2-\|y-y^{k+1}\|^2-\|y^{k+1}-y^k\|^2\right).$$
The second statement follows easily from the arithmetic-geometric mean inequality in Hilbert spaces (see \cite[Proposition 5.3(a)]{shefi-teboulle2014}). 
\end{proof}

A direct consequence of the two inequalities in Lemma \ref{l42-sh-teb-h}  is the following result. 

\begin{lemma}\label{p53-sh-teb-h} In the context of Problem \ref{admm-p1}, assume that $M_1^k-L\id\in{\cal S_+}({\cal H}),
M_1^k\succcurlyeq M_1^{k+1}$, $M_2^k\in{\cal S_+}({\cal G}), M_2^k\succcurlyeq M_2^{k+1}$ for all $k \geq 0$, and let $(x^k,z^k,y^k)_{k\geq 0}$ 
be the sequence generated by Algorithm \ref{alg-h-var}. Then for all $k\geq 0$ and all $(x,z,y)\in{\cal H}\times{\cal G}\times{\cal G}$ the following inequality holds
\begin{eqnarray*}l(x^{k+1},z^{k+1},y)& \leq & l(x,z,y^{k+1})+\frac{c}{2}\left(\|Ax-z^k\|^2-\|Ax-z^{k+1}\|^2\right)\\
&&+\frac{1}{2}\left(\|x-x^k\|_{M_1^k}^2-\|x-x^{k+1}\|_{M_1^{k+1}}^2+\|z-z^k\|_{M_2^k}^2-\|z-z^{k+1}\|_{M_2^{k+1}}^2\right)\\
&&+\frac{1}{2c}\left(\|y-y^k\|^2-\|y-y^{k+1}\|^2\right). 
\end{eqnarray*}
\end{lemma}

We can now state the main result of this subsection.

\begin{theorem}\label{ergodic} In the context of Problem \ref{admm-p1}, assume that $M_1^k-L\id\in{\cal S_+}({\cal H}),
M_1^k\succcurlyeq M_1^{k+1}$, $M_2^k\in{\cal S_+}({\cal G}), M_2^k\succcurlyeq M_2^{k+1}$ for all $k \geq 0$, and let $(x^k,z^k,y^k)_{k\geq 0}$ 
be the sequence generated by Algorithm \ref{alg-h-var}. For all $k \geq 1$ define the ergodic sequences
$$\ol x^k:=\frac{1}{k}\sum_{i=1}^{k}x^{i}, \ \ol y^k:=\frac{1}{k}\sum_{i=1}^{k}y^{i}, \ \ol z^k:=\frac{1}{k}\sum_{i=1}^{k}z^{i}.$$
Then for all $k\geq 1$ and all $(x,z,y)\in{\cal H}\times{\cal G}\times{\cal G}$ it holds
$$l(\ol x^k,\ol z^k,y)-l(x,z,\ol y^k)\leq \frac{\gamma(x,z,y)}{k},$$
where $\gamma(x,z,y):= \frac{c}{2}\|Ax-z^0\|^2
+\frac{1}{2}\left(\|x-x^0\|_{M_1^0}^2+\|z-z^0\|_{M_2^0}^2\right)
+\frac{1}{2c}\|y-y^0\|^2. $
\end{theorem}
 
\begin{proof} We fix $k\geq 1$ and $(x,z,y)\in{\cal H}\times{\cal G}\times{\cal G}$. Summing up the inequalities in Lemma \ref{p53-sh-teb-h} for $i=0,...,k-1$ and using classical arguments for 
telescoping sums, we obtain 
$$\sum_{i=0}^{k-1}l(x^{k+1},z^{k+1},y)\leq \sum_{i=0}^{k-1}l(x,z,y^{k+1})+\gamma(x,z,y).$$
Since $l$ is convex in $(x,z)$ and linear in $y$, the conclusion follows from the definition 
of the ergodic sequences. 
\end{proof}

\begin{remark} Let $(x^*,z^*,y^*)$ be a saddle point for the Lagrangian $l$. By taking $(x,z,y):=(x^*,z^*,y^*)$ in the above theorem it yields 
$$(f+h)(\ol x^k)+g(\ol z^k)+\langle y^*,A\ol x^k-\ol z^k\rangle-\big(f(x^*) + h(x^*) + g(Ax^*) \big) \leq \frac{\gamma(x^*,z^*,y^*)}{k} \ \forall k\geq 1,$$
where $f(x^*) + h(x^*) + g(Ax^*)$ is the optimal objective value of the problem \eqref{prim-h}. Hence, if we suppose that the set of optimal solutions of the dual problem \eqref{dual-h} is contained in a bounded set,  there exists 
$R>0$ such that for all $k \geq 1$
\begin{align*}
(f+h)(\ol x^k)+g(\ol z^k)+R\|A\ol x^k-\ol z^k\|-\big(f(x^*) + h(x^*) + g(Ax^*) \big) & \leq\\ 
\frac{1}{k} \left(\frac{c}{2}\|Ax^*-z^0\|^2
+\frac{1}{2}\|x^*-x^0\|_{M_1^0}^2+\frac{1}{2}\|z^*-z^0\|_{M_2^0}^2
+\frac{1}{c}(R^2+\|y^0\|^2)\right). &
\end{align*}

The set of dual optimal solutions of \eqref{dual-h} is equal to the convex subdifferential of the infimal value function of the problem \eqref{prim-h}
$$\psi : {\cal G} \rightarrow \overline \R, \ \psi(y) = \inf_{x \in {\cal H}} \left(f(x) + h(x) + g(Ax+y) \right),$$
at $0$. This set is weakly compact, thus bounded, if $0 \in \inte(\dom \psi) = \inte(A(\dom f) - \dom g)$ (see \cite{bauschke-book, b-hab, Zal-carte}).

\end{remark}

\subsection{Convergence of the sequence of generated iterates}\label{sub32}

In this subsection we will address  the convergence of the sequence of iterates generated by Algorithm \ref{alg-h-var}. One of the important tools for the proof of the convergence result will be
the following version of the Opial Lemma formulated in the context of variable metrics (see \cite[Theorem 3.3]{combettes-vu2013}). 

\begin{lemma}\label{opial-var} Let $S$ be a nonempty subset of ${\cal H}$ and $(x^k)_{k \geq 0}$ a sequence in ${\cal H}$.  Let $\alpha>0$ and $W^k\in{\cal P}_{\alpha}({\cal H})$ be such that 
$W^k\succcurlyeq W^{k+1}$ for all $k\geq 0$. Assume that:  

(i) for all $z\in S$ and for all $k\geq 0$: $\|x^{k+1}-z\|_{W^{k+1}}\leq \|x^k-z\|_{W^k}$; 

(ii) every weak sequential cluster point of $(x^k)_{k\geq 0}$ belongs to $S$. 

\noindent Then $(x^k)_{k\geq 0}$ converges weakly to an element in $S$. 
\end{lemma}

The proof of the  convergence result relies on techniques specific to monotone operator theory and does not make use of the values of the objective function or of the Lagrangian $l$. This makes it different from the proofs in \cite{shefi-teboulle2014} and from the other conventional convergence proofs for ADMM methods.

\begin{theorem}\label{cong-it} In the context of Problem \ref{admm-p1}, assume that  the set of saddle points of the Lagrangian $l$ is nonempty and that
$M_1^k-\frac{L}{2}\id\in{\cal S_+}({\cal H}),
M_1^k\succcurlyeq M_1^{k+1}$, $M_2^k\in{\cal S_+}({\cal G}), M_2^k\succcurlyeq M_2^{k+1}$ for all $k\geq 0$, and let $(x^k,z^k,y^k)_{k\geq 0}$ be the sequence generated by Algorithm \ref{alg-h-var}.
If one of the following assumptions
\begin{itemize}
\item[(I)] there exists $\alpha_1>0$ such that $M_1^k-\frac{L}{2}\id\in{\cal P}_{\alpha_1}({\cal H})$ for all $k\geq 0$;

\item[(II)] there exists $\alpha, \alpha_2>0$ such that $M_1^k-\frac{L}{2}\id+A^*A\in {\cal P}_{\alpha}({\cal H})$ and $M_2^k\in {\cal P}_{\alpha_2}({\cal G})$ 
for all $k\geq 0$;
\item[(III)] there exists $\alpha>0$ such that $M_1^k-\frac{L}{2}\id+A^*A\in {\cal P}_{\alpha}({\cal H})$ and 
$2M_2^{k+1}\succcurlyeq M_2^{k}\succcurlyeq M_2^{k+1}$ for all $k\geq 0$;
\end{itemize}
is fulfilled, then $(x^k,z^k,y^k)_{k\geq 0}$ converges weakly to a saddle point of the Lagrangian $l$.  
\end{theorem}

\begin{proof} Let $S\subseteq {\cal H}\times {\cal G}\times {\cal G}$ denote the set of the saddle points of the Lagrangian $l$ and $(x^*,z^*,y^*)$ be a fixed element in $S$. Then $z^*=Ax^*$ and the optimality conditions hold
$$-A^*y^*-\nabla h(x^*)\in\partial f(x^*), \  y^*\in\partial g(Ax^*). $$
Let $k\geq 0$ be fixed. 
Taking into account \eqref{opt-cond-f}, \eqref{opt-cond-g} and the monotonicity of $\partial f$ and $\partial g$, we obtain 
$$\langle cA^*(z^k-Ax^{k+1}-c^{-1}y^k)+M_1^k(x^k-x^{k+1})-\nabla h(x^k)+A^*y^*+\nabla h(x^*),x^{k+1}-x^*\rangle\geq 0$$
and 
$$\langle c(Ax^{k+1}-z^{k+1}+c^{-1}y^k)+M_2^k(z^k-z^{k+1})-y^*,z^{k+1}-Ax^*\rangle\geq 0.$$

We consider first the case $L>0$. By the Baillon-Haddad Theorem (see \cite[Corollary 18.16]{bauschke-book}), the gradient of $h$ is $L^{-1}$-cocoercive, hence 
the following inequality holds 
$$\langle\nabla h(x^*)-\nabla h(x^k),x^*-x^k\rangle\geq L^{-1}\|\nabla h(x^*)-\nabla h(x^k)\|^2.$$

Summing up the three inequalities obtained above we get
\begin{align*}
c\langle z^k-Ax^{k+1},Ax^{k+1}-Ax^*\rangle+\langle y^*-y^k,Ax^{k+1}-Ax^*\rangle+\langle\nabla h(x^*)-\nabla h(x^k),x^{k+1}-x^*\rangle &\\
+\langle M_1^k(x^k-x^{k+1}),x^{k+1}-x^*\rangle+c\langle Ax^{k+1}-z^{k+1},z^{k+1}-Ax^*\rangle+\langle y^k-y^*,z^{k+1}-Ax^*\rangle&\\
+\langle M_2^k(z^k-z^{k+1}),z^{k+1}-Ax^*\rangle+\langle\nabla h(x^*)-\nabla h(x^k),x^*-x^k\rangle-L^{-1}\|\nabla h(x^*)-\nabla h(x^k)\|^2 & \geq 0.
\end{align*}
Further, by taking into account \eqref{h-var-y}, it holds
$$\langle y^*-y^k,Ax^{k+1}-Ax^*\rangle+\langle y^k-y^*,z^{k+1}-Ax^*\rangle=\langle y^*-y^k,Ax^{k+1}-z^{k+1}\rangle=
c^{-1}\langle y^*-y^k,y^{k+1}-y^k\rangle.$$
By using some expressions of the inner products in terms of norms, we obtain
\begin{align*}
\frac{c}{2}\left(\|z^k-Ax^*\|^2-\|z^k-Ax^{k+1}\|^2-\|Ax^{k+1}-Ax^*\|^2\right) & \\
+ \frac{c}{2}\left(\|Ax^{k+1}-Ax^*\|^2-\|Ax^{k+1}-z^{k+1}\|^2-\|z^{k+1}-Ax^*\|^2\right) & \\
+\frac{1}{2c}\left(\|y^*-y^k\|^2+\|y^{k+1}-y^k\|^2-\|y^{k+1}-y^*\|^2\right)&\\
+\frac{1}{2}\left(\|x^k-x^*\|_{M_1^k}^2-\|x^k-x^{k+1}\|_{M_1^k}^2-\|x^{k+1}-x^*\|_{M_1^k}^2\right)&\\
+\frac{1}{2}\left(\|z^k-Ax^*\|_{M_2^k}^2-\|z^k-z^{k+1}\|_{M_2^k}^2-\|z^{k+1}-Ax^*\|_{M_2^k}^2\right)&\\
+\langle\nabla h(x^*)-\nabla h(x^k),x^{k+1}-x^k\rangle-L^{-1}\|\nabla h(x^*)-\nabla h(x^k)\|^2& \geq 0.
\end{align*}
By using again relation \eqref{h-var-y} for expressing $Ax^{k+1}-z^{k+1}$ and by taking into account that
\begin{align*}
\langle\nabla h(x^*)-\nabla h(x^k),x^{k+1}-x^k\rangle-L^{-1}\|\nabla h(x^*)-\nabla h(x^k)\|^2 & =\\
-L\left\|L^{-1}\left(\nabla h(x^*)-\nabla h(x^k)\right)+\frac{1}{2}\left(x^k-x^{k+1}\right)\right\|^2+\frac{L}{4}\|x^k-x^{k+1}\|^2&,
\end{align*}
it yields
\begin{align*}\frac{1}{2}\|x^{k+1}-x^*\|_{M_1^k}^2+\frac{1}{2}\|z^{k+1}-Ax^*\|_{M_2^k+c\id}^2+\frac{1}{2c}\|y^{k+1}-y^*\|^2 & \leq\\
\frac{1}{2}\|x^k-x^*\|_{M_1^k}^2+\frac{1}{2}\|z^k-Ax^*\|_{M_2^k+c\id}^2+\frac{1}{2c}\|y^k-y^*\|^2 & \\
-\frac{c}{2}\|z^k-Ax^{k+1}\|^2-\frac{1}{2}\|x^k-x^{k+1}\|_{M_1^k}^2-\frac{1}{2}\|z^k-z^{k+1}\|_{M_2^k}^2 &\\
-L\left\|L^{-1}\left(\nabla h(x^*)-\nabla h(x^k)\right)+\frac{1}{2}\left(x^k-x^{k+1}\right)\right\|^2+\frac{L}{4}\|x^k-x^{k+1}\|^2&
\end{align*}
and from here, by using the monotonicity assumptions on $(M_1^k)_{k \geq 0}$ and $(M_2^k)_{k \geq 0}$, we finally get
\begin{align}\frac{1}{2}\|x^{k+1}-x^*\|_{M_1^{k+1}}^2+\frac{1}{2}\|z^{k+1}-Ax^*\|_{M_2^{k+1}+c\id}^2+\frac{1}{2c}\|y^{k+1}-y^*\|^2 & \leq \nonumber\\
\frac{1}{2}\|x^k-x^*\|_{M_1^k}^2+\frac{1}{2}\|z^k-Ax^*\|_{M_2^k+c\id}^2+\frac{1}{2c}\|y^k-y^*\|^2 & \nonumber \\
-\frac{c}{2}\|z^k-Ax^{k+1}\|^2-\frac{1}{2}\|x^k-x^{k+1}\|_{M_1^k-\frac{L}{2}\id}^2-\frac{1}{2}\|z^k-z^{k+1}\|_{M_2^k}^2 & \nonumber\\
-L\left\|L^{-1}\left(\nabla h(x^*)-\nabla h(x^k)\right)+\frac{1}{2}\left(x^k-x^{k+1}\right)\right\|^2&.\label{fey-var}
\end{align}

In case $L=0$, similar arguments lead to the inequality 
\begin{align}\frac{1}{2}\|x^{k+1}-x^*\|_{M_1^{k+1}}^2+\frac{1}{2}\|z^{k+1}-Ax^*\|_{M_2^{k+1}+c\id}^2+\frac{1}{2c}\|y^{k+1}-y^*\|^2 & \leq \nonumber\\
 \frac{1}{2}\|x^k-x^*\|_{M_1^k}^2+\frac{1}{2}\|z^k-Ax^*\|_{M_2^k+c\id}^2+\frac{1}{2c}\|y^k-y^*\|^2 &\nonumber\\
 -\frac{c}{2}\|z^k-Ax^{k+1}\|^2-\frac{1}{2}\|x^k-x^{k+1}\|_{M_1^k}^2-\frac{1}{2}\|z^k-z^{k+1}\|_{M_2^k}^2.& \label{fey-var-h0}
\end{align}

It is easy to see, by using arguments invoking telescoping sums,  that, in both cases, \eqref{fey-var} and \eqref{fey-var-h0} yield 
\begin{equation}\label{series}\sum_{k\geq 0}\|z^k-Ax^{k+1}\|^2<+\infty, \ \sum_{k\geq 0}\|x^k-x^{k+1}\|_{M_1^k-\frac{L}{2}\id}^2<+\infty, \ \sum_{k\geq 0}\|z^k-z^{k+1}\|_{M_2^k}^2<+\infty.\end{equation}

{\it The case when Assumption (I) is valid}. 

By neglecting the negative terms from the right-hand side of both \eqref{fey-var} and \eqref{fey-var-h0},
it follows that the first assumption in the Opial Lemma (Lemma \ref{opial-var}), when applied in the 
product space ${\cal H}\times {\cal G}\times {\cal G}$, for the sequence 
$(x^k,z^k,y^k)_{k\geq 0}$, for $W^k:= (M_1^k,M_2^k+c\id,c^{-1}\id)$ for $k \geq 0$, and for
$S\subseteq {\cal H}\times {\cal G}\times {\cal G}$ the set of saddle points of the Lagrangian $l$, holds.

Since $M_1^k-\frac{L}{2}\id\in{\cal P}_{\alpha_1}({\cal H})$ for all $k \geq 0$ with $\alpha_1>0$, we get 
\begin{equation}\label{x-x-0} x^k-x^{k+1}\rightarrow 0 \ (k\rightarrow+\infty)
\end{equation}
and \begin{equation}\label{z-Ax-0} z^k-Ax^{k+1}\rightarrow 0 \ (k\rightarrow+\infty).
\end{equation}
A direct consequence of \eqref{x-x-0} and \eqref{z-Ax-0} is 
\begin{equation}\label{z-z-0} z^k-z^{k+1}\rightarrow 0 \ (k\rightarrow+\infty). 
\end{equation}
From \eqref{h-var-y}, \eqref{z-Ax-0} and \eqref{z-z-0} we derive 
\begin{equation}\label{y-y-0} y^k-y^{k+1}\rightarrow 0 \ (k\rightarrow+\infty). 
\end{equation}

The relations \eqref{x-x-0}-\eqref{y-y-0} will play an essential role in the verification of the second assumption in the Opial Lemma. 
Let $(\ol x,\ol z,\ol y)\in {\cal H}\times {\cal G}\times {\cal G}$ be such that 
there exists $(k_n)_{n\geq 0}$, $k_n\rightarrow +\infty$ (as $n\rightarrow +\infty$), and $(x^{k_n},z^{k_n}, y^{k_n})$ converges weakly to 
$(\ol x,\ol z,\ol y)$ (as $n\rightarrow +\infty$). 

From \eqref{x-x-0}  we obtain that $(Ax^{k_n+1})_{n\in\N}$ converges weakly 
to $A\ol x$ (as $n\rightarrow +\infty$), which combined with \eqref{z-Ax-0} yields $\ol z=A\ol x$. 
We use now the following notations for all $n\geq 0$
\begin{align*} 
a_n^* := & \ cA^*(z^{k_n}-Ax^{k_n+1}-c^{-1}y^{k_n})+M_1^{k_n}(x^{k_n}-x^{k_n+1})+\nabla h(x^{k_n+1})-\nabla h(x^{k_n})\\
a_n := & \ x^{k_n+1}\\
b_n^*:= &\  y^{k_n+1}+M_2^{k_n}(z^{k_n}-z^{k_n+1})\\
b_n:= & \ z^{k_n+1}.
\end{align*}
From \eqref{opt-cond-f} and \eqref{opt-cond-g2} we have for all $n\geq 0$
\begin{equation}\label{a-a}a_n^*\in\partial (f+h)(a_n)\end{equation}
and
\begin{equation}\label{b-b}b_n^*\in\partial g(b_n).\end{equation}
Furthermore, from \eqref{x-x-0} we have 
\begin{equation}\label{an-w}a_n \mbox{ converges weakly to } \ol x \ (\mbox{as }n\rightarrow+\infty). 
\end{equation}
From \eqref{y-y-0} and \eqref{z-z-0} we obtain 
\begin{equation}\label{bn-w}b_n^* \mbox{ converges weakly to } \ol y \ (\mbox{as }n\rightarrow+\infty). 
\end{equation}
Moreover, \eqref{h-var-y} and \eqref{y-y-0} yield 
\begin{equation}\label{Aan-bn}Aa_n-b_n \mbox{ converges strongly to } 0 \ (\mbox{as }n\rightarrow+\infty). 
\end{equation}
Finally, we have 
\begin{align*}
a_n^*+A^*b_n^*= & \ cA^*(z^{k_n}-Ax^{k_n+1})+A^*(y^{k_n+1}-y^{k_n})+ \! M_1^{k_n}(x^{k_n}-x^{k_n+1}) +\!A^*M_2^{k_n}(z^{k_n}-z^{k_n+1})\\
& + \nabla h(x^{k_n+1})-\nabla h(x^{k_n}).
\end{align*}
By using the fact that $\nabla h$ is Lipschitz continuous, from \eqref{x-x-0}-\eqref{y-y-0} we get 
\begin{equation}\label{an-Abn}a_n^*+A^*b_n^* \mbox{ converges strongly to } 0 \ (\mbox{as }n\rightarrow+\infty). 
\end{equation}
Taking into account the relations \eqref{a-a}-\eqref{an-Abn} and applying\cite[Proposition 2.4]{a-comb-s} to the operators $\partial (f+h)$ and 
$\partial g$, we conclude that \begin{equation*}-A^*\ol y\in\partial (f+h)(\ol x)=\partial f(\ol x)+\nabla h(\ol x)\mbox{ and }\ol y\in\partial g(A\ol x),\end{equation*}
hence $(\ol x,\ol z,\ol y)=(\ol x,A\ol x,\ol y)$ is a saddle point of the Lagrangian $l$, thus the second assumption of 
the Opial Lemma is verified, too. In conclusion,  $(x^k,z^k,y^k)_{k\geq 0}$ converges weakly to a saddle point of the Lagrangian $l$. 

{\it The case when Assumption (II) is valid}. 

We show that the relations \eqref{x-x-0}-\eqref{y-y-0} are fulfilled also in this case. Indeed, Assumption (II) allows to derive from \eqref{series} that \eqref{z-Ax-0} and \eqref{z-z-0} hold. 
From \eqref{h-var-y}, \eqref{z-Ax-0} and \eqref{z-z-0} we obtain \eqref{y-y-0}. Finally, 
the inequalities 
\begin{align*}\label{ineq}\alpha\|x^{k+1}-x^k\|^2\leq   & \|x^{k+1}-x^k\|_{M_1^k - \frac{L}{2}\id}^2 + \|Ax^{k+1}-Ax^k\|^2 \nonumber \\
\leq & \|x^{k+1}-x^k\|_{M_1^k - \frac{L}{2}\id}^2  + 2\|Ax^{k+1}-z^k\|^2+2\|z^k-Ax^k\|^2 \ \forall k\geq 0\end{align*}
yield \eqref{x-x-0}.

On the other hand, notice that both \eqref{fey-var} and \eqref{fey-var-h0} yield 
\begin{equation}\label{exist-lim}  \exists\lim_{k\rightarrow+\infty}\left(\frac{1}{2}\|x^k-x^*\|_{M_1^k}^2+
\frac{1}{2}\|z^k-z^*\|_{M_2^k+c\id}^2+\frac{1}{2c}\|y^k-y^*\|^2\right),
\end{equation}
hence $(y^k)_{k\geq 0}$ and $(z^k)_{k\geq 0}$ are bounded. Combining this with \eqref{h-var-y} and the condition 
imposed on $M_1^k-\frac{L}{2}\id+A^*A$, we derive that $(x^k)_{k\geq 0}$ is bounded, too. Hence there exists a weakly convergent subsequence 
of $(x^k,z^k,y^k)_{k\geq 0}$. By using the same arguments as in the proof of (I), it follows that
every weak sequential cluster point of $(x^k,z^k,y^k)_{k\geq 0}$ is a saddle point of the Lagrangian $l$. 

Now we show that the set of weak sequential cluster points of $(x^k,z^k,y^k)_{k\geq 0}$ is a singleton. 
Let $(x_1,z_1,y_1),(x_2,z_2,y_2)$ be two such weak sequential cluster points. Then there exist $(k_p)_{p\geq 0}, (k_q)_{q\geq 0}$, 
$k_p\rightarrow+\infty$ (as $p\rightarrow+\infty$), $k_q\rightarrow+\infty$ (as $q\rightarrow+\infty$), a subsequence
$(x^{k_p},z^{k_p}, y^{k_p})_{p \geq 0}$ which converges weakly to $(x_1,z_1,y_1)$ (as $p\rightarrow+\infty$), and a subsequence
$(x^{k_q},z^{k_q}, y^{k_q})_{q \geq 0}$ which converges weakly to $(x_2,z_2,y_2)$ (as $q\rightarrow+\infty$). As seen, $(x_1,z_1,y_1)$ and 
$(x_2,z_2,y_2)$ are saddle points of the Lagrangian $l$ and $z_i=Ax_i$ for $i\in\{1,2\}$. From \eqref{exist-lim}, which is true for every 
saddle point of the Lagrangian $l$, we derive 
\begin{equation}\label{exist-lim1}
\exists\lim_{k\rightarrow+\infty}\left(E(x^k,z^k,y^k; x_1,z_1,y_1)-E(x^k,z^k,y^k; x_2,z_2,y_2)\right),
\end{equation}
where, for $(x^*,z^*,y^*)$ the expression $E(x^k,z^k,y^k; x^*,z^*,y^*)$ is defined as 
$$E(x^k,z^k,y^k; x^*,z^*,y^*)=\frac{1}{2}\|x^k-x^*\|_{M_1^k}^2+
\frac{1}{2}\|z^k-z^*\|_{M_2^k+c\id}^2+\frac{1}{2c}\|y^k-y^*\|^2.$$
Further, we have for all $k \geq 0$
$$\frac{1}{2}\|x^k-x_1\|_{M_1^k}^2-\frac{1}{2}\|x^k-x_2\|_{M_1^k}^2=\frac{1}{2}\|x_2-x_1\|_{M_1^k}^2
+\langle x^k-x_2,M_1^k(x_2-x_1)\rangle,$$
$$\frac{1}{2}\|z^k-z_1\|_{M_2^k+c\id}^2\!\!-\frac{1}{2}\|z^k-z_2\|_{M_2^k+c\id}^2 \!\!=\frac{1}{2}\|z_2-z_1\|_{M_2^k+c\id}^2
\!\!+\langle z^k-z_2, (M_2^k+c\id)(z_2-z_1)\rangle,$$
and
$$\frac{1}{2c}\|y^k-y_1\|^2-\frac{1}{2c}\|y^k-y_2\|^2=\frac{1}{2c}\|y_2-y_1\|^2
+\frac{1}{c}\langle y^k-y_2, y_2-y_1\rangle.$$
Applying \cite[Th\'{e}or\`{e}ème 104.1]{rn}, there exists $M_1\in{\cal S_+}({\cal H})$ such that 
$(M_1^k)_{k \geq 0}$ converges to $M_1$ in the strong operator topology, i.e., $\|M_1^k x - M_1 x\| \to 0$ for all $x \in \mathcal{H}$ (as $k \to +\infty$). Similarly, 
the monotonicity condition imposed on $(M_2^k)_{k \geq 0}$ implies that $\sup_{k\geq 0}\|M_2^k+c\id\|<+\infty$. 
Thus, according to \cite[Lemma 2.3]{combettes-vu2013}, there exists $\alpha'>0$  and $M_2\in {\cal P}_{\alpha'}({\cal G})$ 
such that $(M_2^k+c\id)_{k \geq 0}$ converges to $M_2$ in the strong operator topology (as $k\rightarrow+\infty$). 

Taking the limit in \eqref{exist-lim1} along the subsequences $(k_p)_{p\geq 0}$ and $(k_q)_{q\geq 0}$ and using the last three relations 
above we obtain 
$$\frac{1}{2}\|x_1-x_2\|_{M_1}^2+\langle x_1-x_2,M_1(x_2-x_1)\rangle+ \frac{1}{2}\|z_1-z_2\|_{M_2}^2+\langle z_1-z_2, M_2(z_2-z_1)\rangle$$
$$+\frac{1}{2c}\|y_1-y_2\|^2+\frac{1}{c}\langle y_1-y_2, y_2-y_1\rangle
=\frac{1}{2}\|x_1-x_2\|_{M_1}^2+\frac{1}{2}\|z_1-z_2\|_{M_2}^2+\frac{1}{2c}\|y_1-y_2\|^2,$$
hence $$-\|x_1-x_2\|_{M_1}^2-\|z_1-z_2\|_{M_2}^2-\frac{1}{c}\|y_1-y_2\|^2=0.$$
From here we get $\|x_1-x_2\|_{M_1} =0$, $z_1=z_2$ and $y_1=y_2$. Since
$$\left(\alpha + \frac{L}{2}\right) \|x_1-x_2\|^2 \leq  \|x_1-x_2\|_{M_1}^2 + \|Ax_1-Ax_2\|^2,$$
we obtain that $x_1=x_2$. In conclusion, $(x^k,z^k,y^k)_{k\geq 0}$ converges weakly to a saddle point of the Lagrangian $l$. 

{\it The case when Assumption (III) is valid}. 

Under Assumption (III) we can further refine the inequalities in \eqref{fey-var} and \eqref{fey-var-h0}. Let $k\geq 1$ be fixed. By considering the relation \eqref{opt-cond-g2} for consecutive iterates 
and by taking into account the monotonicity of $\partial g$ we derive 
$$\langle z^{k+1}-z^k,y^{k+1}-y^k+M_2^k(z^k-z^{k+1})-M_2^{k-1}(z^{k-1}-z^k)\rangle\geq 0,$$
hence 
\begin{align}\label{coniter}
\langle z^{k+1}-z^k,y^{k+1}-y^k\rangle & \geq \|z^{k+1}-z^k\|_{M_2^k}^2+\langle z^{k+1}-z^k,M_2^{k-1}(z^{k-1}-z^k)\rangle \nonumber\\
& \geq \|z^{k+1}-z^k\|_{M_2^k}^2-\frac{1}{2}\|z^{k+1}-z^k\|_{M_2^{k-1}}^2-\frac{1}{2}\|z^{k}-z^{k-1}\|_{M_2^{k-1}}^2.
\end{align}
Using that $y^{k+1}-y^k=c(Ax^{k+1}-z^{k+1})$, the last inequality yields
\begin{align}
 \|z^{k+1}-z^k\|_{M_2^k}^2-\frac{1}{2}\|z^{k+1}-z^k\|_{M_2^{k-1}}^2-\frac{1}{2}\|z^{k}-z^{k-1}\|_{M_2^{k-1}}^2 & \leq\nonumber\\
 \frac{c}{2}\left(\|z^k-Ax^{k+1}\|^2-\|z^{k+1}-z^k\|^2-\|Ax^{k+1}-z^{k+1}\|^2\right).\label{ineq-aha}
\end{align}

In case $L>0$, adding \eqref{ineq-aha} and \eqref{fey-var} leads to
\begin{align*}\frac{1}{2}\|x^{k+1}-x^*\|_{M_1^{k+1}}^2+\frac{1}{2}\|z^{k+1}-Ax^*\|_{M_2^{k+1}+c\id}^2+\frac{1}{2c}\|y^{k+1}-y^*\|^2 +
\frac{1}{2}\|z^{k+1}-z^k\|_{3M_2^k-M_2^{k-1}}^2& \leq \nonumber\\
 \frac{1}{2}\|x^{k}-x^*\|_{M_1^{k}}^2+\frac{1}{2}\|z^k-Ax^*\|_{M_2^k+c\id}^2+\frac{1}{2c}\|y^k-y^*\|^2+\frac{1}{2}\|z^k-z^{k-1}\|_{M_2^{k-1}}^2&\nonumber\\
 -\frac{1}{2}\|x^{k+1}-x^k\|_{M_1^k-\frac{L}{2}\id}^2
 -\frac{c}{2}\|z^{k+1}-z^k\|^2 - \frac{1}{2c}\|y^{k+1}-y^k\|^2&\nonumber\\
 -L\left\|L^{-1}\left(\nabla h (x^*)-\nabla h (x^k)\right)+\frac{1}{2}\left(x^k-x^{k+1}\right)\right\|^2. 
\end{align*}
Taking into account that, according to Assumption (III), $3M_2^k-M_2^{k-1}\succcurlyeq M_2^k$, we can conclude that for all $k \geq 1$ it holds
\begin{align}\frac{1}{2}\|x^{k+1}-x^*\|_{M_1^{k+1}}^2+\frac{1}{2}\|z^{k+1}-Ax^*\|_{M_2^{k+1}+c\id}^2+\frac{1}{2c}\|y^{k+1}-y^*\|^2 +
\frac{1}{2}\|z^{k+1}-z^k\|_{M_2^k}^2& \!\!\leq \nonumber\\
 \frac{1}{2}\|x^{k}-x^*\|_{M_1^{k}}^2+\frac{1}{2}\|z^k-Ax^*\|_{M_2^k+c\id}^2+\frac{1}{2c}\|y^k-y^*\|^2+\frac{1}{2}\|z^k-z^{k-1}\|_{M_2^{k-1}}^2&\nonumber\\
 -\frac{1}{2}\|x^{k+1}-x^k\|_{M_1^k-\frac{L}{2}\id}^2-\frac{c}{2}\|z^{k+1}-z^k\|^2 - \frac{1}{2c}\|y^{k+1}-y^k\|^2.& \label{fey-var-M10-III'}
\end{align}

Similarly, in case $L=0$ we obtain 
\begin{align}\frac{1}{2}\|x^{k+1}-x^*\|_{M_1^{k+1}}^2+\frac{1}{2}\|z^{k+1}-Ax^*\|_{M_2^{k+1}+c\id}^2+\frac{1}{2c}\|y^{k+1}-y^*\|^2 +
\frac{1}{2}\|z^{k+1}-z^k\|_{M_2^k}^2& \!\!\leq \nonumber\\
 \frac{1}{2}\|x^{k}-x^*\|_{M_1^{k}}^2+\frac{1}{2}\|z^k-Ax^*\|_{M_2^k+c\id}^2+\frac{1}{2c}\|y^k-y^*\|^2+\frac{1}{2}\|z^k-z^{k-1}\|_{M_2^{k-1}}^2&\nonumber\\
 -\frac{1}{2}\|x^{k+1}-x^k\|_{M_1^k}^2-\frac{c}{2}\|z^{k+1}-z^k\|^2 - \frac{1}{2c}\|y^{k+1}-y^k\|^2.& \label{fey-var-h0-M10-III'}
\end{align}
Using telescoping sum arguments, we obtain that $\|x^{k+1}-x^k\|_{M_1^k-\frac{L}{2}\id} \rightarrow 0$, $y^k - y^{k+1} \rightarrow 0$ 
and $z^k - z^{k+1} \rightarrow 0$ as $k \rightarrow +\infty$. Using \eqref{h-var-y}, it follows that
$A(x^{k}-x^{k+1}) \rightarrow 0$ as $k \rightarrow +\infty$, which, combined with the fact that  $M_1^k-\frac{L}{2}\id + A^*A \in {\cal P}_{\alpha}({\cal {H}})$, for all $k \geq 0$, yields $x^{k}-x^{k+1} \rightarrow 0$ as $k \rightarrow +\infty$.
Consequently, $z^k - Ax^{k+1} \rightarrow 0$ as $k \rightarrow +\infty$. Hence, the relations \eqref{x-x-0}-\eqref{y-y-0} are fulfilled. 
On the other hand, from both \eqref{fey-var-M10-III'} and \eqref{fey-var-h0-M10-III'} we derive 
\begin{equation*}
\exists\lim_{k\rightarrow+\infty}\left(\frac{1}{2}\|x^{k}-x^*\|_{M_1^{k}}^2+\frac{1}{2}\|z^k-Ax^*\|_{M_2^k+c\id}^2+\frac{1}{2c}\|y^k-y^*\|^2+\frac{1}{2}\|z^k-z^{k-1}\|_{M_2^{k-1}}^2\right).
\end{equation*}
By using that 
$$\|z^k-z^{k-1}\|_{M_2^{k-1}}^2 \leq \|z^k-z^{k-1}\|_{M_2^{0}}^2 \leq \|M_2^0\| \|z^k-z^{k-1}\|^2 \ \forall k \geq 1,$$
it follows that $\lim_{k\rightarrow+\infty} \|z^k-z^{k-1}\|_{M_2^{k-1}}^2 = 0$, which further implies that \eqref{exist-lim} holds. 
From here the conclusion follows by arguing as in the proof provided  in the setting of Assumption (II).
\end{proof}

\begin{remark} Choosing as in Remark \ref{rem-cond}, $M_1^k= \frac{1}{\tau_k}\id - cA^*A$, with $\tau_k >0$ and such that $\tau:=\sup_{k \geq 0} \tau_k \in \R$, and $M_2^k = 0$ for all $k \geq 0$, we have
$$\left\langle x,\left(M_1^k-\frac{L}{2}\id\right)x\right\rangle \geq \left(\frac{1}{\tau_k}-c\|A\|^2-\frac{L}{2}\right)\|x\|^2  \geq \left(\frac{1}{\tau}-c\|A\|^2-\frac{L}{2}\right)\|x\|^2 \ \forall x \in {\cal H},$$
which means that under the assumption $\frac{1}{\tau}-c\|A\|^2>\frac{L}{2}$ 
(which recovers the one in Algorithm 3.2 and Theorem 3.1 in \cite{condat2013}), the operators $M_1^k-\frac{L}{2}\id$ 
belong for all $k \geq 0$ to the class ${\cal P}_{\alpha_1}({\cal H})$, with $\alpha_1:= \frac{1}{\tau}-c\|A\|^2-\frac{L}{2}>0$. 
\end{remark}

\begin{remark} By taking $h=0$ and $L=0$, and in each iteration constant operators $M_1^k = M_1 \succcurlyeq 0$ and $M_2^k = M_2 \succcurlyeq 0$ for all $k \geq 0$, Theorem \ref{cong-it} in the context of Assumption (I) covers the first situation investigated in \cite[Theorem 5.6]{shefi-teboulle2014}, where in finite dimensional spaces the matrix $M_1$ was assumed to be positive definite and the matrix $M_2$ to be positive semidefinite. 

The arguments used in \cite[Theorem 5.6]{shefi-teboulle2014} for proving convergence in the case 
when $M_1=0$ and $A$ has full column rank contain flaws and rely on incorrect statements. Theorem  
\ref{cong-it} provides in the context of Assumption (III) (for $h=0$, $L=0$, $M_1^k=0$ and $M_2^k = M_2 \succcurlyeq 0$ for all $k \geq 0$)  the correct proof of this result.

Finally, we notice that the convergence theorem for the iterates of the classical ADMM algorithm (which corresponds to the situation when 
$h=0$,  $L=0$, $M_1=M_2=0$ and $A$ has full column rank, see for example 
\cite{ecb}) is covered by Theorem \ref{cong-it} in the context of Assumption (III).
\end{remark}


\end{document}